\documentclass[a4paper]{amsart}
\usepackage{amsmath,amsthm,amssymb,latexsym,epic,bbm,comment,mathrsfs}
\usepackage{graphicx,enumerate,stmaryrd,color}
\usepackage[all,2cell]{xy}
\xyoption{2cell}

\newtheorem{theorem}{Theorem}
\newtheorem{lemma}[theorem]{Lemma}
\newtheorem{corollary}[theorem]{Corollary}
\newtheorem{proposition}[theorem]{Proposition}

\newtheorem{example}[theorem]{Example}

\usepackage[all]{xy}
\usepackage[active]{srcltx}
\usepackage[parfill]{parskip}
\usepackage{enumerate}

\begin{document}
\title[Simple modules for affine Lie algebras]
{Simple modules for untwisted affine Lie algebras
induced from nilpotent loop subalgebras}


\author[V.~Mazorchuk]{Volodymyr Mazorchuk}

\begin{abstract}
We construct large families  of simple modules for
untwisted affine Lie algebras using induction 
from one-dimensional  modules over 
nilpotent loop subalgebras.
We also show that the vector space of the first
self-extensions for these  module has uncountable
dimension and that generic tensor products of
these modules  are simple.
\end{abstract}

\maketitle

\section{Introduction and description of the results}\label{s1}

Let $\mathfrak{g}$ be a semi-simple complex finite dimensional Lie
algebra with a fixed triangular decomposition
\begin{displaymath}
\mathfrak{g}=\mathfrak{n}_-\oplus \mathfrak{h}\oplus \mathfrak{n}_+. 
\end{displaymath}
As defined in \cite{Ko}, a {\em Whittaker} $\mathfrak{g}$-module
is a module that is generated by a one-dimensional 
$\mathfrak{n}_+$-submodule. Various classes of simple Whittaker modules 
in the above sense were studied and classified by several authors,
see \cite{Ko,Mc1,Mc2,MiSo,Ba}.

Analogues of Whittaker modules can be defined and studied in many
similar situations, see \cite{OW,Ch,LWZ,LG,BCW,ALZ,HY,Chen,AT,CCM} 
and references therein.  A general setup of {\em Whittaker pairs} in which
one can define analogues of Whittaker modules was proposed in
\cite{BM}, see also \cite{MZ2}. A Whittaker pair 
$(\mathfrak{g},\mathfrak{n})$ consists of a Lie algebra $\mathfrak{g}$
and a quasi-nilpotent  subalgebra $\mathfrak{n}$ of $\mathfrak{g}$ such that 
the action of $\mathfrak{n}$ on the adjoint $\mathfrak{n}$-module 
$\mathfrak{g}/\mathfrak{n}$ is locally nilpotent.

Various types of Whittaker modules were studied for infinite-dimensional
Kac-Moody Lie algebras, see \cite{Ch,ALZ,AT,CF,GZ}. In this paper we 
are interested in Whittaker modules for the Whittaker pair
$(\hat{\mathfrak{g}},L(\mathfrak{n}))$, where 
$\hat{\mathfrak{g}}$ is the untwisted affine Lie algebra associated
to a simple finite dimensional Lie algebra $\mathfrak{g}$ and
$L(\mathfrak{n})$ is the loop algebra of the nilpotent part
$\mathfrak{n}$ in a fixed parabolic subalgebra $\mathfrak{p}$  of
$\mathfrak{g}$. Note that $L(\mathfrak{n})$ is nilpotent and
the adjoint action of $L(\mathfrak{n})$ on
$\hat{\mathfrak{g}}/L(\mathfrak{n})$ is locally nilpotent.
Hence $(\hat{\mathfrak{g}},L(\mathfrak{n}))$ is a Whittaker pair
in the sense of \cite{BM}.

The main contribution of this paper is an explicit  construction of 
a large family of new simple $\hat{\mathfrak{g}}$-modules. These
modules are simple Whittaker modules for the above Whittaker pair
and they are indexed by certain (collections of) elements in the dual space to a
vector space of countable dimension. In particular, as a vector space,
the set of our parameters has uncountable dimension.
We note that construction of new classes of
simple modules over affine Lie algebras is a very popular subject,
see \cite{Cha,CP1,CP2,Fu,FT,FK,CTZ,GKMOS} and references therein.

The main result of the paper is Theorem~\ref{thmmain7}  which asserts 
that, under some mild assumptions on the input parameters,
the induced Whittaker module for Whittaker pair
$(\hat{\mathfrak{g}},L(\mathfrak{n}))$ is simple. 
This result and its proof can be found in Section~\ref{s7}.
Additionally, the latter section contains Corollary~\ref{corext}
that establishes the fact that the space of first self-extensions
for the modules we  construct has uncountable dimension.
In Section~\ref{s5} we prove a fairly unexpected result 
that, generically, finite  
tensor products of modules constructed in Theorem~\ref{thmmain7}
are again simple, see Theorem~\ref{thmmain5}.
Section~\ref{s2} contains necessary preliminaries and
Section~\ref{s6} a number of concrete examples.

\subsection*{Acknowledgements}
This research is partially supported by the Swedish Research Council.
The author thanks Rekha Biswal for stimulating discussions,
especially for asking the question whether tensor products of
Whittaker modules are simple or not.
\vspace{5mm}

\section{Setup and preliminaries}\label{s2}

\subsection{Combinatorics of linear duals for 
vector spaces of countable dimension}\label{s3.1}

In what follows, we work over $\mathbb{C}$, in particular,
$\otimes$ means $\otimes_{\mathbb{C}}$.

Let $V$ be a vector space with basis $\mathbf{v}:=\{v_i\,:\,i\in\mathbb{Z}\}$.
Then the dual space $V^*$ can be naturally identified with 
$\mathbb{C}^\mathbb{Z}$, where 
$\mathbf{a}:=(a_i)_{i\in\mathbb{Z}}\in\mathbb{C}^\mathbb{Z}$ corresponds to
$\varphi_{\mathbf{a}}:V\to\mathbb{C}$ which sends $v_i$ to $a_i$, for $i\in\mathbb{Z}$.

The group $\mathbb{Z}$ acts both on $V$ and on $V^*$ by shifting the 
indices: $n\cdot v_i:=v_{i+n}$ and $n\cdot (a_i):=(a_{i+n})$, for
$n,i\in\mathbb{Z}$. We will use the notation
$x^{(n)}$ to denote the $n$-translate $n\cdot x$ of $x$.

The space $\mathbb{C}^{\mathbb{Z}}$ has a
$\mathbb{Z}$-invariant subspace $\mathbb{C}^{\mathbb{Z}}_{\mathrm{fin}}$
consisting of all $\mathbf{a}$ in which only finitely many coefficients
are non-zero. Note that $\mathbb{C}^{\mathbb{Z}}_{\mathrm{fin}}$ is exactly
the linear span of the {\em dual basis $\{v_i^*\,:\,i\in\mathbb{Z}\}$} which 
is, indeed, a basis in this subspace.

An element $\mathbf{a}\in\mathbb{C}^\mathbb{Z}$ will be called 
\begin{itemize}
\item {\bf generic} provided that the multiset of all $\mathbb{Z}$-translates of 
$\mathbf{a}$ is linearly independent;
\item and {\bf strongly generic} provided that the union of the multiset of all 
$\mathbb{Z}$-translates  of $\mathbf{a}$ with the element
$(ia_i)_{i\in\mathbb{Z}}$ is linearly independent.
\end{itemize}
An alternative point of view is that $\mathbf{a}$ is not generic if and only if 
there exists a non-zero  $v\in V$ such that $\mathbf{a}$ vanishes on all 
$\mathbb{Z}$-translates of $V$. 

\begin{example}\label{example1}
{\em
Any constant element in $\mathbb{C}^{\mathbb{Z}}$, that is any element
in which all coefficients are the same, is not generic.
The element $(\delta_{i,0})_{i\in\mathbb{Z}}$, where $\delta$ is the
Kronecker symbol, is generic but  not strongly generic.
} 
\end{example}

\begin{lemma}\label{lemma1}
All non-zero elements in $\mathbb{C}^{\mathbb{Z}}_{\mathrm{fin}}$
are generic.
\end{lemma}

\begin{proof}
The zero element is, clearly, not generic. On the other hand,
if $\mathbf{a}\in \mathbb{C}^{\mathbb{Z}}_{\mathrm{fin}}$ is not zero
and $v\in V$ is not zero, we can find a $\mathbb{Z}$-translate $w$ of
$v$ for which there is a unique coordinate $i\in\mathbb{Z}$
such that the corresponding coefficients both in $\mathbf{a}$
and in $w$ are non-zero. The evaluation of $\mathbf{a}$ at this
$w$ is, clearly, non-zero. Hence such $\mathbf{a}$ is generic.
\end{proof}

Let us also make the following observation:
for a non-zero $v=\sum_i c_i v_i\in V$, where all $c_i\in\mathbb{C}$
and only finitely many of them are non-zero, set
\begin{displaymath}
l(v)=\min\{i:c_i\neq 0\},\qquad
r(v)=\max\{i:c_i\neq 0\}
\end{displaymath}
and $\omega(v):=r(v)-l(v)\in\mathbb{Z}_{\geq 0}$. 
If $\mathbf{a}\in\mathbb{C}^\mathbb{Z}$ is not generic and vanishes
on all $\mathbb{Z}$-translates of the above $v\in V$, then
$\mathbf{a}$ is uniquely determined by $a_0,a_1,\dots,a_{\omega(v)-1}$.
Indeed, if $\omega(v)=0$, then $\mathbf{a}$ is the zero element.
If $\omega(v)>0$ and $a_0,a_1,\dots, a_{\omega(v)-1}$ are fixed, we take the
$\mathbb{Z}$-translate $v'$ of $v$ for which $l(v')=0$
and thus $r(v')=\omega(v)$. The condition that $\mathbf{a}$ vanishes on
$v'$ determines $a_{\omega(v)}$ uniquely. Translating $v'$ to the right,
we similarly determine all $a_i$, for $i> \omega(v)$. Translating  $v'$ 
to the left, we similarly determine all $a_i$, for $i<0$.

We will use the notation $\langle{}_-,{}_-\rangle$ to denote the
evaluation of an element of $V^*$ at an element of $V$. Expressed
using the coordinates with respect to the standard bases, this is 
given by the formula $\sum_i x_iy_i$. We note that the latter
expression is well-defined provided that either $(x_i)$ or
$(y_i)$ has only finitely many non-zero coefficients.
For simplicity, our convention is that $\langle{}_-,{}_-\rangle$
is symmetric.

For $\mathbf{a}\in\mathbb{C}^\mathbb{Z}$, denote by
$\mathbf{a}^\perp$ the set of all $v\in V$ such that
$\langle\mathbf{a},v^{(i)}\rangle=0$, for all $i\in\mathbb{Z}$.
Clearly, $\mathbf{a}^\perp$ is a subspace of $V$. 
If $\mathbf{a}$ is zero, then the corresponding $\mathbf{a}^\perp$
is the whole $V$. If $\mathbf{a}$ is generic, then 
the corresponding $\mathbf{a}^\perp$ is the zero subspace of $V$.

The following lemma is inspired by the notion of {\em size}
of a non-homogeneous ideal in \cite[Page~221]{MZ}.

\begin{lemma}\label{lemma2}
If  $\mathbf{a}$ is not generic, then $\mathbf{a}^\perp$ is a subspace of
$V$ of finite codimension $m$. Moreover, 
this number $m$ is the minimal non-negative integer for  which there is
$v\in V$ whose all $\mathbb{Z}$-translates are annihilated by
$\mathbf{a}$ and for which $\omega(v)=m$. Additionally, 
the collection of all $\mathbb{Z}$-translates of this $v$ forms a basis
in $\mathbf{a}^\perp$ and this basis can be extended to a basis of 
$V$ using $v_0,v_1,\dots,v_{\omega(v)-1}$ (note that this latter set
is empty if $\omega(v)=0$).
\end{lemma}

\begin{proof}
Define $m$ as the minimal non-negative integer for  which there is
$v\in V$ whose all $\mathbb{Z}$-translates are annihilated by
$\mathbf{a}$ and for which $\omega(v)=m$. Let $v'\in \mathbf{a}^\perp$
be non-zero. Assume $\omega(v')>\omega(v)$. Then we can take a 
$\mathbb{Z}$-translate $v''$ of $v'$ such that $l(v)=l(v'')$. 
Subtracting from $v''$ a non-zero multiple of $v$ we can eliminate
the $l(v'')$-component and create a new vector $v'''$ with
$\omega(v''')<\omega (v')$.

If $\omega(v')=\omega(v)$, then a similar manipulation must result
in $0$, due to the minimality of $m$. This implies that $v'$
is a linear combination of $\mathbb{Z}$-translates of $v$.
All claims of the lemma follow.
\end{proof}

For a non-generic $\mathbf{a}$, we will call the number
$m$ in Lemma~\ref{lemma2} the {\bf size} of $\mathbf{a}$.

A set $Q$ of elements from $\mathbb{C}^\mathbb{Z}$ will be called 
\begin{itemize}
\item {\bf generic} provided that the multiset of all $\mathbb{Z}$-translates of 
all elements in $Q$ is linearly independent;
\item and {\bf strongly generic} provided that the union of the multiset of all 
$\mathbb{Z}$-translates  of all elements in $Q$ with the set of 
all elements of the form $(ia_i)_{i\in\mathbb{Z}}$,
where $\mathbf{a}\in Q$, is linearly independent.
\end{itemize}
The word ``multiset'' used in the definition above means that no
translates of different elements in $Q$ may coincide.

Since $\mathbb{C}^\mathbb{Z}$ has uncountable dimension, finite, or even
countable, subsets of elements in $\mathbb{C}^\mathbb{Z}$ are almost always
strongly generic.

\subsection{Injective modules over the polynomial algebra}\label{s3.1-5}

For $V$ as in the previous subsection, consider the free commutative
algebra $\mathbb{C}[V]$ on $V$, that is the quotient of the free tensor 
algebra on $V$ by the ideal generated by $v_i\otimes v_j-v_j\otimes v_i$,
for all $i,j\in\mathbb{Z}$. The algebra $\mathbb{C}[V]$ is the polynomial
algebra in the variables $v_i$, where $i\in\mathbb{Z}$. Note that 
$\mathbb{C}[V]$ has a natural (positive) $\mathbb{Z}$-grading by monomial degree.

Consider the (left) regular $\mathbb{C}[V]$-module ${}_{\mathbb{C}[V]}\mathbb{C}[V]$.
It is free and hence projective in $\mathbb{C}[V]$-Mod. As a graded module, we have
\begin{displaymath}
\mathbb{C}[V]=\bigoplus_{i\geq 0} \mathbb{C}[V]_i.
\end{displaymath}
Note that $\mathbb{C}[V]_i$ is infinite dimensional, for $i\neq 0$.
Now we can consider the restricted dual space
\begin{displaymath}
\mathbb{C}[V]^{\star}=\bigoplus_{i\geq 0} \mathbb{C}[V]_i^*, 
\end{displaymath}
where $*$, as usual, denotes $\mathrm{Hom}_\mathbb{C}({}_-,\mathbb{C})$.
The module $\mathbb{C}[V]^{\star}$ is a $\mathbb{C}[V]$-submodule of 
the dual module $\mathbb{C}[V]^*$. In fact,
directly from the construction it follows that $\mathbb{C}[V]^{\star}$
is the indecomposable injective envelope of the (unique) simple graded 
$\mathbb{C}[V]$-module. In particular, the vector space of
the first self-extensions of this
(unique) simple graded $\mathbb{C}[V]$-module is naturally isomorphic to
$\mathbb{C}[V]_1^*$ and thus has uncountable dimension.

For each $\mathbf{a}\in\mathbb{C}^\mathbb{Z}$, we have an algebra automorphism
$\alpha_\mathbf{a}$ of $\mathbb{C}[V]$ which sends $v_i$ to $v_i-a_i$, for
$i\in\mathbb{Z}$. Applying twisting by $\alpha_\mathbf{a}$ to the above grading,
produces a new grading on $\mathbb{C}[V]$ and gives a new indecomposable injective
object ${}^{\alpha_\mathbf{a}}\mathbb{C}[V]^{\star}$ which now lives 
in the corresponding category of graded modules.
Consequently, the space of the first self-extensions of any simple finite dimensional 
$\mathbb{C}[V]$-module has uncountable dimension.

\subsection{Untwisted affine Lie algebras}\label{s2.1}

Let $\mathfrak{g}$ be a semi-simple finite dimensional Lie algebra 
with a fixed triangular decomposition
\begin{displaymath}
\mathfrak{g}=\mathfrak{n}_-\oplus \mathfrak{h}\oplus \mathfrak{n}_+.
\end{displaymath}
Consider the loop algebra $L(\mathfrak{g})$ of $\mathfrak{g}$,
defined as $\mathfrak{g}\otimes\mathbb{C}[t,t^{-1}]$, where the 
Lie bracket is given by
\begin{displaymath}
[x\otimes t^i,y\otimes t^j]=[x,y]\otimes t^{i+j}. 
\end{displaymath}
The associated {\em untwisted affine Lie algebra} 
$\hat{\mathfrak{g}}$ is defined as 
$\mathfrak{g}\oplus \mathbb{C}d\oplus \mathbb{C}c$,
where $d$ is the derivation on $L(\mathfrak{g})$ given by
$[d,x\otimes t^i]=i(x\otimes t^i)$ and $c$ is a central element
involved in the universal central extension of the loop algebra
given by 
\begin{displaymath}
[x\otimes t^i,y\otimes t^j]=[x,y]\otimes t^{i+j}
+\delta_{i,-j}(x,y)c, 
\end{displaymath}
where $({}_-,{}_-)$ is the Killing form on $\mathfrak{g}$.

Let $\mathfrak{p}$ be a fixed parabolic subalgebra of $\mathfrak{g}$
containing the Borel subalgebra $\mathfrak{h}\oplus \mathfrak{n}_+$
and different from $\mathfrak{g}$.
We can write
\begin{displaymath}
\mathfrak{p}=\mathfrak{l}\oplus \mathfrak{n}, 
\end{displaymath}
where $\mathfrak{n}$ is the nilradical of $\mathfrak{p}$
and $\mathfrak{l}$ is a (reductive) Levi subalgebra. We can extend
this to a decomposition of $\mathfrak{g}$:
\begin{displaymath}
\mathfrak{g}=\mathfrak{m}\oplus\mathfrak{l}\oplus \mathfrak{n}, 
\end{displaymath}
where $\mathfrak{m}$ is the negative counterpart of $\mathfrak{n}$.

The loop algebra $L(\mathfrak{n})$ of $\mathfrak{n}$ is nilpotent and
the adjoint action of $L(\mathfrak{n})$ on $\hat{\mathfrak{g}}$
is locally nilpotent. Hence $(\hat{\mathfrak{g}},L(\mathfrak{n}))$ 
is a Whittaker pair in the sense of \cite{BM}.

Let $\Phi$ be the root system of $\mathfrak{g}$. Denote by 
$\Phi_{\mathfrak{n}}$ the subset of $\Phi$ consisting of all
roots that appear in $\mathfrak{n}$. We further write
$\Phi_{\mathfrak{n}}=\Phi_{\mathfrak{n}}^0\cup \Phi_{\mathfrak{n}}^1$,
where $\Phi_{\mathfrak{n}}^1$  is the subset of 
$\Phi_{\mathfrak{n}}$ consisting of all roots that appear in 
$[\mathfrak{n},\mathfrak{n}]$ and
$\Phi_{\mathfrak{n}}^0:=\Phi_{\mathfrak{n}}\setminus \Phi_{\mathfrak{n}}^1$.
The condition that $\mathfrak{p}$ is different from $\mathfrak{g}$
means that $\Phi_{\mathfrak{n}}\neq \varnothing$. As 
$\mathfrak{n}$ is nilpotent, it follows that 
$\Phi_{\mathfrak{n}}^0\neq \varnothing$.

Finally, for each root $\alpha\in \Phi$ we fix a non-zero element
$X_\alpha$ in the corresponding root subspace $\mathfrak{g}_\alpha$
in $\mathfrak{g}$.

\section{Simple induced Whittaker modules for 
$(\hat{\mathfrak{g}},L(\mathfrak{n}))$}\label{s7}

\subsection{Setup}\label{s7.1}

For $\theta\in\mathbb{C}$, consider the quotient 
$U_\theta$ of the universal enveloping algebra $U(\hat{\mathfrak{g}})$
of $\hat{\mathfrak{g}}$ by the two-sided ideal generated
by $c-\theta$. Note that $U_\theta$ is a simple algebra
if $\theta\neq 0$, see \cite{BS}.

For a map $\Lambda$ from $\Phi_{\mathfrak{n}}^0$ to $\mathbb{C}^\mathbb{Z}$,
we can consider the corresponding one-dimensional 
$L(\mathfrak{n})$-module $\mathbb{C}_\Lambda$. This module is 
annihilated by $[L(\mathfrak{n}),L(\mathfrak{n})]$ and, for
$\alpha\in \Phi_{\mathfrak{n}}^0$, the element $X_\alpha\otimes t^j$
acts on it as $a_j$, where $\mathbf{a}=(a_i)_{i\in\mathbb{Z}}=\Lambda(\alpha)$
and $j\in\mathbb{Z}$.

\subsection{The main result}\label{s7.2}

\begin{theorem}\label{thmmain7}
Assume that $\Lambda$ and $\theta$ are as above and that 
the image of $\Lambda$ (considered as a multiset) is strongly generic.
Then the module
\begin{displaymath}
M(\Lambda,\theta)=
U_\theta\bigotimes_{U(L(\mathfrak{n}))}\mathbb{C}_\Lambda
\end{displaymath}
is simple.
\end{theorem}

The fact that we consider the image of $\Lambda$ 
as a multiset means that, in order to be strongly 
generic, $\Lambda(\alpha)$ can coincide neither with 
any $\mathbb{Z}$-translate of $\Lambda(\beta)$
nor  with $(ia_i)_{i\in\mathbb{Z}}$, where $\Lambda(\beta)=(a_i)_{i\in\mathbb{Z}}$,
for any two different elements  $\alpha,\beta\in\Phi_{\mathfrak{n}}^0$.

Clearly, $M(\Lambda,\theta)\cong M(\Lambda',\theta')$ implies
$\Lambda=\Lambda'$ and $\theta=\theta'$. Therefore
Theorem~\ref{thmmain7} produces a huge family of pairwise non-isomorphic
simple $U(\hat{\mathfrak{g}})$-modules.

\subsection{Proof of the main result}\label{s7.3}

Let $\Lambda$ and $\theta$ be as above and assume that  the image of $\Lambda$ 
(considered as a multiset) is strongly generic. We are going to 
prove that $1\otimes 1$ is the only element of $M(\Lambda,\theta)$
(up to scalars) on which all elements of $L(\mathfrak{n})$ act
as scalars. Since the action of $L(\mathfrak{n})$ on 
$M(\Lambda,\theta)$ is locally finite and $1\otimes 1$ generates
$M(\Lambda,\theta)$, this immediately implies simplicity of 
$M(\Lambda,\theta)$.

Consider the lower central series 
\begin{displaymath}
\mathfrak{n}=\mathfrak{n}_0\supset 
\mathfrak{n}_1\supset \mathfrak{n}_2\supset \dots
\end{displaymath}
where $\mathfrak{n}_i=[\mathfrak{n},\mathfrak{n}_{i-1}]$,
and let $X_i$ be the subset of $\Phi$ consisting of all
those roots of $\mathfrak{n}_i$ that are not roots of
$\mathfrak{n}_{i+1}$. Note that $X_0=\Phi^0_\mathfrak{n}$
and $X_1\cup X_2\cup\dots =\Phi^1_\mathfrak{n}$.
As $\mathfrak{n}$ is nilpotent, $X_i=\varnothing$, for
$i\gg 0$. Let $k$ be maximal such that $X_k\neq\varnothing$.

Fix some basis for the root system of the semi-simple part of
$\mathfrak{l}$. This defines the standard partial order on
the set of all roots of $\mathfrak{l}$, that is 
$\alpha\leq \beta$ provided that $\beta-\alpha$ is a 
linear combination of positive roots with non-negative
integer coefficients. Fix a linear order on
$(\Phi\cup\{0\})\setminus\Phi_\mathfrak{n}$ such that
\begin{itemize}
\item the elements in $-X_k$ are all smaller than 
the elements in $-X_{k-1}$ which, in turn, are all smaller
that the element in $-X_{k-2}$ and then all the way up
with the roots of $\mathfrak{l}$ and $0$ coming last;
\item within each $-X_i$ and within the roots of 
$\mathfrak{l}$ and $0$, one element is greater than the other
if their difference is a linear combination of positive roots
of $\mathfrak{l}$ with non-negative integer coefficients.
\end{itemize}

Now define a linear order on all elements of the form
$d$, $X_\alpha\otimes t^j$, where $\alpha\in\Phi$, and 
$H_\alpha\otimes t^j$, where $\alpha$ is a simple root,
by
\begin{itemize}
\item first comparing the corresponding roots with respect to
the above order;
\item  then comparing the corresponding exponents of $t$;
\item and, finally, letting $d$ be greater  than all
$H_\alpha\otimes t^j$.
\end{itemize}

By the PBW Theorem, the module $M(\Lambda,\theta)$ has
a basis given by standard monomials in the above elements.
Consider the lexicographic order on the set of these
monomials. Here $1$ is the maximum element, it corresponds
to $1\otimes 1\in M(\Lambda,\theta)$ and this element 
has the property that the action of all elements from
$L(\mathfrak{n})$ on it is scalar.

Now let $\omega$ be an element of $M(\Lambda,\theta)$ that is not
a scalar multiple of $1\otimes 1$. Let $\mathbf{l}(\omega)$ be the
{\bf leading term} of $\omega$, that is the minimum element in the
set of all monomials that appear in $\omega$ with non-negative
coefficients. Without loss of generality, we may assume that 
the coefficient at $\mathbf{l}(\omega)$ for $\omega$ is $1$.

In order to prove our theorem, we want to prove that
the $L(\mathfrak{n})$-submodule generated by $\omega$
differs from $\mathbb{C}\omega$. For this it is enough to
show that there is $\alpha\in\Phi_\mathfrak{n}^0$ and
$j\in\mathbb{Z}$ such that the action of the element
$(X_\alpha\otimes t^j)-a_j$, where $\Lambda(\alpha)=(a_i)_{i\in\mathbb{Z}}$, 
on $\omega$ is not zero.

Assume $\mathbf{l}(\omega)=u_1u_2\dots u_m$, where each 
$u_i$ is a basis element as above (i.e. either $d$ or 
$X_\alpha\otimes t^j$ or $H_\alpha\otimes t^j$).
We will consider several cases depending on what $u_m$ is.

{\bf Case~1.} Assume that $u_m$ corresponds to $\beta\in -X_i$,
for $i>0$. Then $\beta=\gamma-\alpha$, for some
$\gamma\in -X_{i-1}$ and some $\alpha\in X_0=\Phi_\mathfrak{n}^0$.
Let $\Lambda(\alpha)=(a_i)_{i\in\mathbb{Z}}$
and let us apply $(X_{\alpha}\otimes t^j)-a_j$, for some chosen $j$.

Due to the Leibniz rule, the effect of this application on any 
monomial $v_1v_2\dots v_p$ is given by the sum of the monomials of
the form 
$v_1v_2\dots v_{s-1}[X_{\alpha}\otimes t^j,v_s]v_{s+1}\dots v_p$.
For simplicity, we can choose $j$ small enough such that $c$ 
never appears in the computation of the commutators.
Consequently, up to a non-zero scalar, the monomial
$u_1u_2\dots u_{m-1}[X_{\alpha}\otimes t^j,u_m]$ will be the
leading term of the outcome. In particular, this outcome is
non-zero.

{\bf Case~2.} Assume that $u_m$ corresponds to $\beta\in -X_0$. 
Here the argument is the same as in Case~1, with the
only difference that $\gamma=0$.

{\bf Case~3.} Assume that $u_m$ corresponds to $\beta$ which is
either $0$ or a weight of $\mathfrak{l}$. This is the most 
complicated situation. The problem is that in this case
$[X_{\alpha}\otimes t^j,u_m]$ belongs to $\mathfrak{n}$
and hence its action on $1\otimes 1$ is scalar. Therefore we
expect $u_1u_2\dots u_{m-1}$ to be the leading term of the
new element. Unfortunately, any monomial of the form
$u_1u_2\dots u_{m-1}u$, where $u\in \mathfrak{l}$, may 
contribute to this new leading term. Therefore we need
to show that we can choose $\alpha$ and $j$ such that
all these contributions do not cancel each other.

It is clear that $u_1u_2\dots u_{m-1}$ does not really play any
role in the argument anymore. So, we are essentially looking
at the situation when $\omega\in L(\mathfrak{l})\oplus \mathbb{C}d$,
so we take this $\omega$, apply $[X_{\alpha}\otimes t^j,\omega]$
and then apply the linear functional given by $\Lambda$.
This defines a linear map from $L(\mathfrak{l})\oplus \mathbb{C}d$
to $\mathbb{C}$. We want to prove that the intersection of the
kernels of all these maps, taken over all $\alpha\in\Phi_\mathfrak{n}^0$
and over all $j\in\mathbb{Z}$ is zero. This is where we will use 
the assumption that the image of $\Lambda$  (considered as a multiset) 
is strongly generic.

To start with, we choose $\alpha$ such that $\beta+\alpha$ is a root.
Now choose a basis in $\mathfrak{h}$ by starting with a basis 
of the codimension one subspace which annihilates $\alpha$ and
then extending it by one element to a basis of  $\mathfrak{h}$.
Denote that last element by $X_0$ (similarly to $X_\alpha$,
where $\alpha$ was a root). The remaining basis elements of the 
Cartan commute with $X_\alpha$ and hence play no role.
To consider all potential cancellations,
we need to look at all $\gamma$, which are roots of 
$\mathfrak{l}$ or $0$, such that $\gamma+\alpha$ is a root.
The cancellations will potentially come from
$[X_\gamma\otimes t^s,X_\alpha\otimes t^j]$ as well as
from $[d,X_\alpha\otimes t^j]$.

For $\gamma$ as above, let $\mathbf{b}^\gamma=(b^\gamma_i)_{i\in\mathbb{Z}}$
be the corresponding vector of coefficients for $\omega$. Note that 
this vector is finitary, that is, only finitely many coefficients are
non-zero. Similarly we define $b^d$ as the coefficient for $d$. Denote
$\Lambda(\alpha)$ by $\mathbf{a}$. For a fixed exponent $j$,
the commutator $[X_\gamma\otimes t^s,X_\alpha\otimes t^j]$,
which we can sum up over all $s$, produces, up to a non-zero scalar,
the coefficient $\langle\mathbf{b}^\gamma,\mathbf{a}^{(j)}\rangle$.
Similarly, the commutator $[d,X_\alpha\otimes t^j]$
produced the coefficient $ja_j b^d$.
The condition that all this should sum up  to zero, for each $j$,
is equivalent to saying that the set of all $\mathbf{a}^{(j)}$,
where $j\in\mathbb{Z}$,
together with the vector $(ia_i)$ is linearly dependent.
This contradicts our assumption and completes the proof.

\subsection{$U(L(\mathfrak{n}))$-socle of $M(\Lambda,\theta)$}\label{s7.4}

Our proof of Theorem~\ref{thmmain7} shows that the
locally finite  $U(L(\mathfrak{n}))$-module $M(\Lambda,\theta)$ 
has simple socle. This module, however, is not the injective hull
of its simple socle in the category of locally finite $U(L(\mathfrak{n}))$-modules.
Indeed, $M(\Lambda,\theta)$ has countable dimension while 
the injective hull of its simple socle does not, see 
Subsection~\ref{s3.1-5}.

In the classical situation of semi-simple finite dimensional 
Lie algebras, generic Whittaker modules coincide with the 
injective hulls of their socles in the category of 
locally finite $U(\mathfrak{n}_+)$-modules. 

The results of the next subsection are inspired by this
fundamental differences between our and the classical situations.
The results of the next section have no analogues in the 
classical case due to elementary dimension counting arguments.

\subsection{First self-extensions}\label{s7.7}

\begin{corollary}\label{corext}
Let $\Lambda$  and $\theta$ be as in Theorem~\ref{thmmain7}.
Then the vector space
\begin{displaymath}
\mathrm{Ext}^1_{U(\hat{\mathfrak{g}})}
(M(\Lambda,\theta),M(\Lambda,\theta))
\end{displaymath}
has uncountable dimension.
\end{corollary}

\begin{proof}
By adjunction, the vector space in question is isomorphic to
\begin{displaymath}
\mathrm{Ext}^1_{L(\mathfrak{n})}
(\mathbb{C}_{\Lambda},\mathrm{Res}^{\hat{\mathfrak{g}}}_{L(\mathfrak{n})}(M(\Lambda,\theta))).
\end{displaymath}
In the proof of Theorem~\ref{thmmain7} we showed that the
$L(\mathfrak{n})$-module 
$\mathrm{Res}^{\hat{\mathfrak{g}}}_{L(\mathfrak{n})}(M(\Lambda,\theta))$
has simple socle.

Applying $\mathrm{Hom}_{L(\mathfrak{n})}
(\mathbb{C}_{\Lambda},{}_-)$ to the short exact sequence
\begin{displaymath}
0\to\mathbb{C}_{\Lambda}\to
\mathrm{Res}^{\hat{\mathfrak{g}}}_{L(\mathfrak{n})}(M(\Lambda,\theta))\to
\mathrm{Coker}\to 0, 
\end{displaymath}
gives rise to  an exact sequence
\begin{equation}\label{eqmn23}
\mathrm{Hom}_{L(\mathfrak{n})}
(\mathbb{C}_{\Lambda},\mathrm{Coker})\to
\mathrm{Ext}^1_{L(\mathfrak{n})}
(\mathbb{C}_{\Lambda},\mathbb{C}_{\Lambda})\to
\mathrm{Ext}^1_{L(\mathfrak{n})}
(\mathbb{C}_{\Lambda},\mathrm{Res}^{\hat{\mathfrak{g}}}_{L(\mathfrak{n})}(M(\Lambda,\theta))).
\end{equation}
Since $M(\Lambda,\theta)$ has countable dimension, so does
$\mathrm{Coker}$. Hence $\mathrm{Hom}_{L(\mathfrak{n})}
(\mathbb{C}_{\Lambda},\mathrm{Coker})$ has countable dimension as well.

In Subsection~\ref{s3.1-5} we noticed that the 
vector space
\begin{displaymath}
\mathrm{Ext}^1_{L(\mathfrak{n})}
(\mathbb{C}_{\Lambda},\mathbb{C}_{\Lambda})
\end{displaymath}
has uncountable dimension. 
The exactness  of \eqref{eqmn23} implies that 
the rightmost term in \eqref{eqmn23} has uncountable dimension.
The claim follows.
\end{proof}

\subsection{The case of loop algebra}\label{s7.5}

We remark that, for the loop algebra $L(\mathfrak{g})$, an analogue 
of Theorem~\ref{thmmain7} holds, with a similar proof, under a weaker 
assumption that the image of $\Lambda$ (considered as a multiset) is generic.

\section{Tensor products of simple Whittaker modules}\label{s5}

\subsection{The result}\label{s5.1}

\begin{theorem}\label{thmmain5}
Assume that $\Lambda$, $\Lambda'$,  $\theta$  and $\theta'$ are as above and that 
the union of the images of $\Lambda$ and $\Lambda'$ (considered as a multiset) 
is strongly generic. Then we have
\begin{displaymath}
M(\Lambda,\theta)\bigotimes_\mathbb{C}  M(\Lambda',\theta')
\cong M(\Lambda+\Lambda',\theta+\theta'),
\end{displaymath}
in particular, the tensor product on the left hand side is a simple  module.
\end{theorem}

Of course, not all tensor products of modules
of the form $M(\Lambda,\theta)$ are simple.
For example, it is easy to check that the module
$M(\Lambda,\theta)\bigotimes_\mathbb{C}  M(-\Lambda,-\theta)$ 
is not simple.

\subsection{Proof of Theorem~\ref{thmmain5}}\label{s5.2}

We start by observing that we have an isomorphism
$\mathbb{C}_\Lambda\bigotimes_\mathbb{C}\mathbb{C}_{\Lambda'}\cong
\mathbb{C}_{\Lambda+\Lambda'}$ of $L(\mathfrak{n})$-modules.
Consequently, $L(\mathfrak{n})$ acts locally finitely on the module
$M(\Lambda,\theta)\bigotimes_\mathbb{C}  M(\Lambda',\theta')$
with unique eigenvalues of the element of $L(\mathfrak{n})$
given by $\Lambda+\Lambda'$. It is also clear that $c$ acts on 
$M(\Lambda,\theta)\bigotimes_\mathbb{C}  M(\Lambda',\theta')$
as $\theta+\theta'$.

We denote by $v$ and $w$ the canonical generators of 
$M(\Lambda,\theta)$ and $M(\Lambda',\theta')$,  respectively.
Just like in our proof of Theorem~\ref{thmmain7}, it is enough to show that
the only element of  $M(\Lambda,\theta)\bigotimes_\mathbb{C}  M(\Lambda',\theta')$
on which $L(\mathfrak{n})$ acts via scalars is $v\otimes w$, up to scalars.
Similarly to the proof of Theorem~\ref{thmmain7}, we take some other element
in $M(\Lambda,\theta)\bigotimes_\mathbb{C}  M(\Lambda',\theta')$, and show
that there is an element in $L(\mathfrak{n})$ whose application,
minus the corresponding scalar,  gives a non-zero outcome.

We have a natural basis of $M(\Lambda,\theta)\bigotimes_\mathbb{C}  M(\Lambda',\theta')$
given by pairs of monomials as in the proof of Theorem~\ref{thmmain7}.
We order them lexicographically, by comparing first the left components and
then the right components. Now the same argument as in 
the proof of Theorem~\ref{thmmain7} work, except that we have to consider
one extra case.

{\bf Extra Case.} Assume that we have an element of the form
$((u_1u_2\dots u_m)\cdot v)\otimes w$ and that $u_m$ corresponds to $\beta$ which is
either $0$ or a weight of $\mathfrak{l}$. Then the application of 
$L(\mathfrak{n})$ produces a scalar multiple of 
$((u_1u_2\dots u_{m-1})\cdot v)\otimes w$. The latter can also  
appear after the application of $L(\mathfrak{n})$
to $((u_1u_2\dots u_{m-1})\cdot v)\otimes (u\cdot w)$, where
$u$ corresponds to $\beta$ which is either $0$ or a weight of $\mathfrak{l}$. 
Therefore we have to  argue that one can avoid cancellation of the coefficient
at $((u_1u_2\dots u_{m-1})\cdot v)\otimes w$. This is done by the same argument
as in the last case of the proof of Theorem~\ref{thmmain7}, since we
have assumed that the union of the images of $\Lambda$ and $\Lambda'$ 
(considered as a multiset)  is strongly generic.
The claim of the theorem follows.

\section{Examples}\label{s6}

\subsection{Construction of (strongly) generic sets}\label{s6.1}

It is not too difficult to produce large examples of generic and strongly generic sets in
$\mathbb{C}^\mathbb{Z}$. For instance, for $j\in \mathbb{R}_{>1}$, define
$\mathbf{a}(j)=(a(j)_i)_{i\in\mathbb{Z}}$ as follows:
\begin{displaymath}
a(j)_i:=
\begin{cases}
0, & i\leq 0;\\
j^i, & i>0.
\end{cases}
\end{displaymath}

\begin{proposition}\label{prop6-1}
The set $\{\mathbf{a}(j)\,:\, j\in \mathbb{R}_{>1}\}$ is strongly generic.
\end{proposition}

\begin{proof}
Consider some non-trivial linear combination of all these vectors, their $\mathbb{Z}$-shifts
and $(ia_i)$-modifications and assume that this linear combination is zero.
Let $j$ be maximal such that the corresponding $(a(j)_i)$
or $(ia(j)_i)$ appears in this linear combination with a non-zero
coefficient. 

Assume first that the coefficient at $(ia(j)_i)$ is non-zero. Note that
$ia(j)_i$ growth strictly faster than all other functions that appear 
with non-zero coefficient. This implies that it outgrowth all other terms
of the linear combination for $i\gg 0$ and hence the corresponding terms
will be non-zero. This is a contradiction which  means that this case
cannot occur.

Consequently, some $\mathbb{Z}$-shifts of $(a(j)_i)$ appear in our linear combination
with non-zero coefficients. Since the linear combination of these terms 
outgrowth the rest, for $i\gg 0$, already the linear combination of these
terms must be zero. But in such linear combination we have the smallest
$\mathbb{Z}$-shift which appears with a non-zero coefficient. Since 
$(a(j)_i)$ is zero up to the zero term and is non-zero after that, 
that first non-zero term for that smallest $\mathbb{Z}$-shift cannot
be cancelled by any other shift. The obtained contradiction completes the proof.
\end{proof}

\subsection{$\mathfrak{sl}_2$-example}\label{s6.2}

Let $\{e,h,f\}$ be the standard basis for $\mathfrak{sl}_2$,
where  $e$ corresponds to the positive root $\alpha$. 
For the algebra $\widehat{\mathfrak{sl}}_2$, we have a unique choice of
$\mathfrak{n}$, namely $\mathbb{C}e$. Then
$L(\mathbb{C}e)$ is a commutative algebra.
Taking any strongly generic $\mathbf{a}$ as $\Lambda(\alpha)$,
for example, any $\mathbf{a}(j)$ from the previous subsection,
and any $\theta$, produces a simple 
$\widehat{\mathfrak{sl}}_2$-module $M(\mathbf{a},\theta)$.

\subsection{$\mathfrak{sl}_3$-example}\label{s6.3}

Let $\{\alpha,\beta,\alpha+\beta\}$ be the set of positive roots for
$\mathfrak{sl}_3$. We have two essentially different cases for 
our construction.

{\bf Case~1.} We can take $\mathfrak{n}=\mathfrak{n}_+$ in which case
$\Phi^0_\mathfrak{n}=\{\alpha,\beta\}$. In this case 
$L(\mathfrak{n})$ is nilpotent of degree two,  but not abelian.
Choosing 
a strongly generic pair consisting of $\mathbf{a}$ (as $\Lambda(\alpha)$)
and $\mathbf{b}$ (as $\Lambda(\beta)$), and any $\theta$,
we get the corresponding simple 
$\widehat{\mathfrak{sl}}_3$-module $M(\mathbf{a},\mathbf{b},\theta)$.
For example,  we can take as $\mathbf{a}$ and $\mathbf{b}$  any two
different elements constructed in the previous subsection.

{\bf Case~2.} We can take $\mathfrak{n}=\mathbb{C}\{X_\alpha,X_{\alpha+\beta}\}$ 
in which case $\Phi^0_\mathfrak{n}=\{\alpha,\alpha+\beta\}$. Note that
$L(\mathfrak{n})$ is abelian in this case. Now, choosing 
a strongly generic pair consisting of $\mathbf{a}$ (as $\Lambda(\alpha)$)
and $\mathbf{c}$ (as $\Lambda(\alpha+\beta)$), and any $\theta$,
we get the corresponding simple 
$\widehat{\mathfrak{sl}}_3$-module $M(\mathbf{a},\mathbf{c},\theta)$.
For example,  we can take as $\mathbf{a}$ and $\mathbf{c}$  any two
different elements constructed in the previous subsection.

\vspace{2mm}

\noindent
V.~M.: Department of Mathematics, Uppsala University, Box. 480,
SE-75106, Uppsala,\\ SWEDEN, email: {\tt mazor\symbol{64}math.uu.se}

\end{document}